\documentclass[smallextended]{svjour3}       
\smartqed  
\usepackage{amsmath,amssymb,txfonts}
\usepackage{comment}

\def\re{\mathbb{R}}
\def\N{\mathbb{N}}

\def\({\left(}
\def\){\right)}
\def\pd{\partial}
\def\lap{\Delta}

\def\ep{\varepsilon}
\def\w{\omega}
\def\la{\lambda}
\def\F{\mathbf{F}}

\def\ol{\overline}

\def\intO{\int_{\Omega}}

\begin{document}

\title{Weighted Hardy's inequality in a limiting case and the perturbed Kolmogorov equation
}


\author{Megumi Sano         \and
        Futoshi Takahashi 
}

\institute{Megumi Sano \at
		Department of Mathematics, Graduate School of Science, Osaka City University, 3-3-138 Sugimoto, Sumiyoshi-ku, Osaka, 558-8585, Japan \\
		\email{megumisano0609@st.osaka-cu.ac.jp} \\ 
           \and
           Futoshi Takahashi \at
		Department of Mathematics, Graduate School of Science \& OCAMI, Osaka City University, 3-3-138 Sugimoto, Sumiyoshi-ku, Osaka, 558-8585, Japan \\
		\email{futoshi@sci.osaka-cu.ac.jp} \\ 
}


\date{\today}

\maketitle

\begin{abstract}

In this paper, we show a weighted Hardy inequality in a limiting case for functions in weighted Sobolev spaces with respect to an invariant measure.
We also prove that the constant in the left-hand side of the inequality is optimal.
As applications, we establish the existence and nonexistence of positive exponentially bounded weak solutions to a parabolic problem involving the Ornstein-Uhlenbeck operator
perturbed by a critical singular potential in two dimensional case, according to the size of the coefficient of the critical potential.
These results can be considered as counterparts in the limiting case of results which established in \cite{GGR(AA)} \cite{Hauer-Rhandi} in the non-critical cases,
and are also considered as extensions of a result in \cite{Cabre-Martel} to the Kolmogorov operator case perturbed by a critical singular potential.

\keywords{Weighted Hardy's inequality \and Limiting case \and Kolmogorov operator \and Singular potential}
\subclass{35A23 \and 26D10.}
\end{abstract}

%
%

\section{Introduction}

Let $\Omega$ be a domain in $\re^N$ with $0 \in \Omega$, $N \ge 1$, $1 < p < \infty$,
$A$ be a real $N \times N$-symmetric positive semi-definite matrix,
and
\begin{align}
\label{rho_A}
	d \mu_A = \rho_A (x) dx \quad \text{with} \quad \rho_A (x)=c \cdot \text{exp} \( -\frac{1}{p} (x^t Ax)^{\frac{p}{2}}\), \,\, x \in \Omega.
\end{align}
Here $c > 0$ is chosen so that $\int_{\re^N} d\mu_A = 1$.
More generally, we consider a Borel probability measure $d\mu = \rho(x) dx$ defined on $\Omega \subseteq \re^N$,
and let $W_{\mu, 0}^{1,p}(\Omega)$ denote a weighted Sobolev space which is a completion of $C_c^{\infty}(\Omega)$ with respect to the (semi-) norm $\| \nabla \cdot \|_{L^p(\Omega; d\mu)}$.

In this paper, we concern the limiting case $p = N \ge 2$.
Let $\Omega \subset \re^N$ be a bounded domain containing the origin, $R=\sup_{x \in \Omega} |x| < \infty$ and $a \ge 1$. 
Let $p = N \ge 2$ in \eqref{rho_A}.
In this paper, first we show the following weighted critical Hardy inequality
\begin{align}
\label{weighted L^N}
	\( \frac{N-1}{N} \)^N \intO \dfrac{|u|^N}{|x|^N (\log \frac{aR}{|x|})^N} d\mu_A 
	&\le \intO \left| \nabla u \cdot \frac{x}{|x|} \right|^N d\mu_A \notag \\
	&+ \( \frac{N-1}{N} \)^{N-1} \intO \dfrac{|u|^N (x^t Ax)^{\frac{N}{2}}}{|x|^N (\log \frac{aR}{|x|})^{N-1}} d\mu_A
\end{align}
holds for all $u$ in $W_{\mu_A, 0}^{1,N}(\Omega)$. 
We also prove that the constant $\( \frac{N-1}{N} \)^N$ in the left-hand side is optimal when $A$ is positive definite.
For a general weight function $\rho = \rho(x)$ satisfying some assumptions, 
we also prove a weighted critical Hardy inequality (with non-optimal constant) on two-dimensional domain.

The limiting case is left to be considered in \cite{GGR(AA)} and \cite{Hauer-Rhandi}. 
Actually in \cite{Hauer-Rhandi},
the authors prove a (non-critical) weighted Hardy inequality 
\begin{align}
\label{weighted L^p}
	\( \frac{|N-p|}{p} \)^p \intO \dfrac{|u|^p}{|x|^p} d\mu_A 
	&\le \intO \left| \nabla u \cdot \frac{x}{|x|} \right|^p d\mu_A \notag \\
	&+ \( \frac{|N-p|}{p} \)^{p-1} {\rm sgn} (N-p) \intO \dfrac{|u|^p (x^t Ax)^{\frac{p}{2}}}{|x|^p} d\mu_A
\end{align}
for functions $u$ in $W_{\mu_A, 0}^{1,p}(\Omega)$ when $1 < p < N$, $N \ge 2$,
and $u$ in $W_{\mu_A, 0}^{1,p}(\Omega \setminus \{ 0 \})$ when $p > N \ge 1$,
where $\mu_A$ is defined in \eqref{rho_A}.
The inequality \eqref{weighted L^p} was first established in \cite{GGR(AA)} when $p=2$ and $N \ge 3$.

Next, by using the optimality of the critical Hardy constant for $\rho_A$ with $p = N$, 
we study the existence and nonexistence of positive weak solutions of a parabolic equation driven by the symmetric Ornstein-Uhlenbeck operator perturbed by a singular potential in dimension $N=2$.
This part can be considered as an extension of \cite{GGR(AA)} to the two dimensional critical case.
Indeed, by using the weighted Hardy inequality \eqref{weighted L^p} for $\mu_A$ with $p=2$ and $N \ge 3$, 
and a result similar to the one in \cite{Cabre-Martel} which is applicable to the Kolmogorov operator
\[
	Lu = \Delta u + \frac{\nabla \rho}{\rho} \cdot \nabla u
\]
with respect to a positive Borel probability measure $d\mu = \rho dx$,
the authors in \cite{GGR(AA)} prove the following result.
 
\begin{theorem}(Goldstein-Goldstein-Rhandi \cite{GGR(AA)})
\label{Thm:GGR(AA)}
Assume $N \ge 3$ and $p=2$ in \eqref{rho_A}.
Let $A$ be a real $N \times N$-symmetric positive semi-definite matrix and $0 \le V(x) \le \frac{c}{|x|^2}$, $x \in \re^N$.
Then the following assertions hold:

(i) If $0 \le c \le \( \frac{N-2}{2} \)^2$, then there exists a weak solution $u \in C([0, \infty), L^2(\re^N; d\mu_A))$ of
\begin{align*}
	\begin{cases}
	\pd_t u(x,t) = \lap u(x,t) - Ax \cdot \nabla u(x,t) + V(x)u(x,t), &\quad t>0, x \in \re^N, \\
	u(x,0) = u_0(x), &\quad x \in \re^N,
	\end{cases}
\end{align*}
satisfying 
\begin{align*}
	\| u(t) \|_{L^2(\Omega; d\mu_A )} \le Me^{\omega t} \| u_0\|_{L^2(\Omega; d\mu_A)}, \quad t \ge 0
\end{align*}
for some constants $M \ge 1, \omega \in \re$ and for any $0 \le u_0 \in L^2(\Omega; d\mu_A)$.

(ii) If $c> \(\frac{N-2}{2} \)^2$, then for any $0 \le u_0 \in L^2(\Omega; d\mu_A) \setminus \{ 0\}$, 
there is no positive weak solution with $V(x)=\frac{c}{|x|^2}$ satisfying the above exponential boundedness.
\end{theorem}

Note that if $\mu = \mu_A$ with $p=2$, then the Kolmogorov operator $L$ is of the form
\[
	L_A u = \Delta u - Ax \cdot \nabla u,
\]
which is known as the symmetric Ornstein-Uhlenbeck operator.
This type of operator arises from many areas of mathematics, such as probability, mathematical physics, and mathematical finance.
Later, Theorem \ref{Thm:GGR(AA)} was generalized by Hauer-Rhandi \cite{Hauer-Rhandi} to the case $p \ne 2$ and $\Omega = (0, \infty) \subset \re$ $(N=1)$
and Goldstein-Hauer-Rhandi \cite{GHR} for the general case. See also \cite{GGR(DCDS)}.
For the classical case $L = \Delta$, the study of existence and nonexistence of positive solutions to the heat equation with a singular potential was initiated by
Baras-Goldstein \cite{Baras-Goldstein} and now enjoys various extensions, see \cite{AGG}, \cite{Goldstein-Kombe}, \cite{Goldstein-Kombe(Positivity)}, \cite{Goldstein-Zhang(JFA)}, \cite{Goldstein-Zhang(TAMS)}, 
and the references therein.

In this paper, we prove the corresponding result for the parabolic problem driven by the symmetric Ornstein-Uhlenbeck operator perturbed by a singular critical potential
of the form $V(x) = \frac{c}{|x|^2(\log \frac{aR}{|x|})^2}$ on two dimensional bounded domains with the Dirichlet boundary conditions.

At the end of this section, we fix several notations:
Let $B^k(R)$ be the $k$-dimensional ball centered at the origin with radius $R$ in $\re^k$.
$B^N(R)$ will be denoted by $B(R)$.
$|B^k(R)|$ denotes the $k$-dimensional volume of $B^k(R)$.
$\w_N$ denotes the area of the unit sphere in $\re^N$.
$(X)_N$ be the $N$-th component of the vector $X \in \re^N$.

%
%
\section{A weighted critical Hardy inequality : $p=N$}

In this section, we prove several weighted Hardy type inequalities for functions in the critical weighted Sobolev space.
Next theorem is a generalization of a result in \cite{Hauer-Rhandi} to the critical case.
Critical Hardy type inequalities with sharp ($a = 1$) and non-sharp ($a > 1$) weight when $d\mu$ is the Lebesgue measure
have been studied by many authors recently,
see for example, \cite{AS}, \cite{Ioku-Ishiwata}, \cite{Sano}, \cite{Sano-TF(EJDE)}, \cite{Sano-TF(CVPDE)}, \cite{TF} and the references therein,
and the sharp critical Hardy inequality was proved originally by Leray in his thesis in 1933 \cite{Leray} when $N=2$.

%
%

\begin{theorem}
\label{Thm:WH}
Let $\Omega$ be a bounded domain in $\re^N$, $N \ge 2$, $R= \sup_{x \in \Omega} |x|$, $a \ge 1$, and $A$ be a real $N \times N$-symmetric positive semi-definite matrix. 
Let $\mu_A$ be defined in \eqref{rho_A} with $p = N$.
Then the inequality
\begin{align}
\label{WH}
	\( \frac{N-1}{N} \)^N \intO \dfrac{|u|^N}{|x|^N (\log \frac{aR}{|x|})^N} d\mu_A &\le \intO \left| \nabla u \cdot \frac{x}{|x|} \right|^N d\mu_A \notag \\
	&+ \( \frac{N-1}{N} \)^{N-1} \intO \dfrac{|u|^N (x^t Ax)^{\frac{N}{2}}}{|x|^N (\log \frac{aR}{|x|})^{N-1}} d\mu_A
\end{align}
holds for all $u \in W_{\mu_A, 0}^{1,N}(\Omega)$. 
Moreover, if $A$ is positive definite and $0 \in \Omega$, then the constant $(\frac{N-1}{N})^N$ in the left-hand side of \eqref{WH} is optimal.
\end{theorem}

\begin{proof} {\it \,of Theorem \ref{Thm:WH}}.
By density, it is enough to show that the inequality (\ref{WH}) holds for all $u \in C_c^1 (\Omega)$.
We fix $\la \ge 0$ and $\beta >0$, which will be chosen later. 
Set
\begin{align*}
	\F (x)=\la \rho_A (x) \frac{x}{|x|^N (\log \frac{aR}{|x|})^\beta}  \quad \text{for} \,\, x \in \Omega \setminus \{ 0\}.
\end{align*}
Here $\rho_A$ is defined in (\ref{rho_A}) with $p=N$. Then we easily check that 
\begin{align*}
	{\rm div}  \F (x) = \la \rho_A (x) \left[ \frac{\beta}{|x|^N (\log \frac{aR}{|x|})^{\beta +1}} - \frac{(x^t Ax)^{\frac{N}{2}}}{|x|^N (\log \frac{aR}{|x|})^{\beta}}  \right] 
	\quad \text{for} \,\, x \in \Omega \setminus \{ 0\}.
\end{align*}
By applying integration by parts and Young's inequality, we have
\begin{align*}
	&\intO |u|^N \la \left[ \frac{\beta}{|x|^N (\log \frac{aR}{|x|})^{\beta +1}} - \frac{(x^t Ax)^{\frac{N}{2}}}{|x|^N (\log \frac{aR}{|x|})^{\beta}}  \right] d\mu_A \\
	&= -N \intO |u|^{N-2}u ( \nabla u \cdot \F ) dx \\
	&=-N \la \intO \frac{|u|^{N-2}u}{|x|^{N-1} (\log \frac{aR}{|x|})^{\beta}} \( \nabla u \cdot \frac{x}{|x|} \) d\mu_A \\
	&\le \intO \left| \nabla u \cdot \frac{x}{|x|} \right|^N d\mu_A + (N-1) \la^{\frac{N}{N-1}} \intO \dfrac{|u|^N}{|x|^N (\log \frac{aR}{|x|})^{\frac{\beta N}{N-1}}} d\mu_A.
\end{align*}
Here note that the left-hand side of the first equality is well-defined because the following properties hold true by the assumption of $A$:
\begin{align*}
	\frac{1}{|x|^N (\log \frac{aR}{|x|})^{\beta +1}} \in L^1_{\text{loc}}(\Omega ) 
	\quad \text{and} \quad \frac{(x^t Ax)^{\frac{N}{2}}}{|x|^N (\log \frac{aR}{|x|})^{\beta}} \le \frac{|A|^{\frac{N}{2}}}{(\log \frac{aR}{|x|})^{\beta}} \in L^1_{\text{loc}}(\Omega ).
\end{align*}
Now, if we choose $\beta =N-1$, then we obtain
\begin{align*}
	(\la \beta -(N-1) \la^{\frac{N}{N-1}} ) \intO \dfrac{|u|^N}{|x|^N (\log \frac{aR}{|x|})^N} d\mu_A 
	\le \intO \left| \nabla u \cdot \frac{x}{|x|} \right|^N d\mu_A + \la \intO \dfrac{|u|^N (x^t Ax)^{\frac{N}{2}}}{|x|^N (\log \frac{aR}{|x|})^{N-1}} d\mu_A.
\end{align*}
Furthermore, if we choose $\la =\( \frac{N-1}{N} \)^{N-1}$ which attains the maximum of the function $\la \mapsto (\la \beta -(N-1) \la^{\frac{N}{N-1}} )$ on the half line $[0, \infty )$, 
then we obtain the inequality (\ref{WH}) for all $u \in C_c^1 (\Omega)$. 
Therefore the inequality (\ref{WH}) also holds for all $u \in W_{\mu_A, 0}^{1,N}(\Omega)$.

%
%
Next we shall show the optimality of the constant $(\frac{N-1}{N})^N$ in (\ref{WH}) if $A$ is positive definite and $0 \in \Omega$. 
To do so, we fix $\la > (\frac{N-1}{N})^N$ and take any $\tilde{\la} \ge 0$. 
Set 
\begin{align*}
	&E(u) = \intO \left| \nabla u \cdot \frac{x}{|x|} \right|^N d\mu_A + \tilde{\la} \intO \dfrac{|u|^N (x^t Ax)^{\frac{N}{2}}}{|x|^N (\log \frac{aR}{|x|})^{N-1}} d\mu_A, \\
	&F_\la (u) = \la \intO \dfrac{|u|^N}{|x|^N (\log \frac{aR}{|x|})^N} d\mu_A
\end{align*}
for $u \in W_{\mu_A, 0}^{1,N}(\Omega) \setminus \{ 0\}$.
Now we consider the test function $\varphi_{\gamma, \ep} \in W_{\mu_A, 0}^{1,N}(B(R))$ given by
\begin{align*}
	\varphi_{\gamma, \ep} (x) = \( \log \frac{aR}{|x|} \)^\gamma \xi_\ep (x),
\end{align*}
where $\gamma < \frac{N-1}{N}$, $\ep>0$ is chosen so that $B(\ep) \subset \Omega$,
and $\xi_\ep \in C_c^{\infty}(B(\ep))$ is a cut-off function with $0 \le \xi_\ep \le 1, \xi_\ep \equiv 1$ on $B(\frac{\ep}{2})$, $|\nabla \xi_\ep | \le B \ep^{-1}$ on $B(\ep)$ for some $B>0$. 
Note that there exist $\alpha_1, \alpha_2 >0$ such that $\alpha_1 |x|^2 \le x^t Ax \le \alpha_2 |x|^2$ for all $x \in \Omega$ because $A$ is positive definite.
Then we have
\begin{align}
\label{deno1}
	F(\varphi_{\gamma, \ep})
	&\ge \la \int_{B(\frac{\ep}{2})} \dfrac{|\varphi_{\gamma, \ep}|^N}{|x|^N (\log \frac{aR}{|x|})^N} \rho_A (x) dx \nonumber \\
	&\ge \la \int_{B(\frac{\ep}{2})} \( \log \frac{aR}{|x|} \)^{\gamma N-N} c\, \text{exp}\( -\frac{\alpha_2^{\frac{N}{2}}}{N} |x|^N \) \frac{dx}{|x|^N} \nonumber \\
	&\ge \la c \, \text{exp}\( -\frac{\alpha_2^{\frac{N}{2}}}{N} \(\frac{\ep}{2} \)^N \) \w_N \int_0^{\frac{\ep}{2}} \( \log \frac{aR}{r} \)^{\gamma N-N} \frac{dr}{r} \nonumber \\
	&=\la c \, \text{exp}\( -\frac{\alpha_2^{\frac{N}{2}}}{N} \( \frac{\ep}{2} \)^N \) \frac{\w_N}{N} \( \log \frac{2aR}{\ep} \)^{N(\gamma -\frac{N-1}{N})} \( \frac{N-1}{N}-\gamma \)^{-1}.
\end{align}
And also we obtain
\begin{align}
\label{derivative1}
	\intO \left| \nabla \varphi_{\gamma, \ep} \cdot \frac{x}{|x|} \right|^N d\mu_A 
	&\le \gamma^N \int_{B(\frac{\ep}{2})} \( \log \frac{aR}{|x|} \)^{\gamma N-N} \rho_A (x) \frac{dx}{|x|^N} + c \int_{B(\ep ) \setminus B(\frac{\ep}{2})} | \nabla \varphi_{\gamma, \ep} |^N dx  \nonumber \\
	&\le \gamma^{N} c \, \frac{\w_N}{N} \( \log \frac{2aR}{\ep} \)^{N(\gamma -\frac{N-1}{N})} \( \frac{N-1}{N}-\gamma \)^{-1} + R(\gamma, \ep),
\end{align}
where $R(\gamma, \ep) = c \int_{B(\ep ) \setminus B(\frac{\ep}{2})} | \nabla \varphi_{\gamma, \ep} |^N dx$.
Note that the remainder term $R(\gamma, \ep)$ can be estimated as follows:
\begin{align*}
	&R(\gamma, \ep) \\
	&\le c\, 2^{N-1} \int_{B(\ep ) \setminus B(\frac{\ep}{2})} \left| \xi_{\ep} \nabla \( \( \log \frac{aR}{|x|} \)^{\gamma} \) \right|^N + \left| \( \log \frac{aR}{|x|} \)^{\gamma} \nabla \xi_{\ep} \right|^N dx \\
	&\le c 2^{N-1} \gamma^N \int_{B(\ep) \setminus B(\frac{\ep}{2})} \( \log \frac{aR}{|x|} \)^{(\gamma -1)N} \frac{dx}{|x|^N} +c 2^{N-1} 
	(B \ep^{-1})^N \int_{B(\ep) \setminus B(\frac{\ep}{2})} \( \log \frac{aR}{|x|} \)^{\gamma N} dx \\
	&\le c 2^{N-1} \w_N  \gamma^N ((\gamma -1)N +1)^{-1} \left[ \( \log \frac{2aR}{\ep} \)^{(\gamma -1)N+1} - \( \log \frac{aR}{\ep} \)^{(\gamma -1)N+1} \right] \\
	&\qquad + c 2^{N-1} \w_N  B^N (\gamma N +1)^{-1} \left[ \( \log \frac{2aR}{\ep} \)^{\gamma N+1} - \( \log \frac{aR}{\ep} \)^{\gamma N+1} \right].
\end{align*}
By applying the mean value theorem for the function $x \mapsto x^p$ for $p=(\gamma -1)N+1$ or $p=\gamma N +1$, there exist positive constants $b, d$ satisfying $\log \frac{aR}{\ep} \le b$ and $d$ $\le \log \frac{2aR}{\ep}$ such that
\begin{align}
\label{remainder}
	|R(\gamma, \ep)| &\le c 2^{N-1} \w_N \log 2 \left[ \gamma^N b^{N(\gamma -\frac{N-1}{N}) -1} + B^N d^{\gamma N} \right] \nonumber \\
	&\le c 2^{N-1} \w_N \log 2 \left[ \gamma^N \( \log \frac{aR}{\ep} \)^{N(\gamma -\frac{N-1}{N}) -1} + B^N \( \log \frac{2aR}{\ep} \)^{\gamma N} \right] \notag \\
	&= O(1) \quad \text{as} \quad \gamma \nearrow \frac{N-1}{N}.
\end{align}
In the same way as above, we also obtain
\begin{align}
\label{AA}
	\intO \dfrac{|\varphi_{\gamma, \ep}|^N (x^t Ax)^{\frac{N}{2}}}{|x|^N (\log \frac{aR}{|x|})^{N-1}} d\mu_A
	&\le |A|^{\frac{N}{2}} c \, \w_N \ep^N (\gamma N-N+2)^{-1} \( \log \frac{aR}{\ep} \)^{\gamma N-N+2} \notag \\
	&= O(1) \quad \text{as} \quad \gamma \nearrow \frac{N-1}{N}.
\end{align}
From (\ref{derivative1}), (\ref{remainder}), and (\ref{AA}), we have
\begin{align}
\label{nume1}
	E(\varphi_{\gamma, \ep}) &\le \gamma^N c \frac{\w_N}{N} \( \log \frac{2aR}{\ep} \)^{N(\gamma -\frac{N-1}{N})} \( \frac{N-1}{N}-\gamma \)^{-1} \notag \\
	&\qquad + o \( \( \frac{N-1}{N}-\gamma \)^{-1} \) \quad \text{as} \,\, \gamma \nearrow \frac{N-1}{N}.
\end{align}
From the estimates (\ref{deno1}) and (\ref{nume1}), 
if we have chosen $\ep >0$ independent of $\gamma$ so small such that $\la \text{exp}\( -\frac{\alpha_2^{\frac{N}{2}}}{N} \(\frac{\ep}{2}\)^N \) > (\frac{N-1}{N})^N > \gamma^N$, 
which is possible since $\la > (\frac{N-1}{N})^N$,
then we observe that
\begin{align*}
	E(\varphi_{\gamma, \ep}) < F_\la (\varphi_{\gamma, \ep})
\end{align*}
for $\gamma$ close to $\frac{N-1}{N}$.
Therefore the inequality (\ref{WH}) never holds if the constant on the left-hand side of (\ref{WH}) is bigger than $(\frac{N-1}{N})^N$.
Hence the constant $(\frac{N-1}{N})^N$ in (\ref{WH}) is optimal.
\end{proof}
\qed

%
%
\begin{remark}
\label{rem:WH}
Let $\tilde{\la} \ge 0$ and $\la > (\frac{N-1}{N})^N$.
Then by using the test function $\varphi_{\gamma, \ep}$, we observe that
\begin{align*}
	\inf_{0 \neq u \in W_{\mu_A, 0}^{1,N}(\Omega )} \dfrac{\intO | \nabla u |^N d\mu_A + \tilde{\la} \intO \frac{|u|^N (x^t Ax)^{\frac{N}{2}}}{|x|^N (\log \frac{aR}{|x|})^{N-1}} d\mu_A 
	- \la \intO \frac{|u|^N}{|x|^N (\log \frac{aR}{|x|})^N} d\mu_A}{\intO |u|^N d\mu_A} = - \infty
\end{align*}
holds true.
\end{remark}

\vspace{1em}

%
%
In the case $N=2$, we can obtain the critical Hardy type inequality (with non-optimal constant) for the general weight function $\rho = \rho(x)$ 
satisfying the following conditions: 
%
%
\begin{align*}
(H1) \,\, &0 < \rho \in H^2(\Omega ) \cap C^{1+\alpha}(\Omega ), \\
(H2) \,\, &\text{for any}\, \ep >0 \, \text{there is} \, C_\ep \in \re \, \text{such that} \left| 
\frac{\nabla \rho}{\rho} \right|^2 - \frac{\lap \rho}{\rho} \le \ep \left| 
\frac{\nabla \rho}{\rho} \right|^2 + C_\ep.
\end{align*}

\begin{proposition}
\label{Prop:WH general}
Let $\Omega$ be a bounded domain in $\re^2$ containing the origin, $R= \sup_{x \in \Omega} |x|$ and $a \ge 1$. 
Let $d\mu = \rho(x) dx$ and assume $(H1)$ and $(H2)$ are satisfied.
Then for any $\delta >0$, there exists $C_\delta >0$ such that the inequality
\begin{align}\label{WH general}
	\frac{1}{4} \intO \dfrac{|\phi |^2}{|x|^2 (\log \frac{aR}{|x|})^2} d\mu \le (4+\delta ) \intO |\nabla \phi |^2 d\mu + C_\delta \intO |\phi |^2 d\mu
\end{align}
holds for all $\phi \in W^{1,2}_{\mu, 0}(\Omega)$.
\end{proposition}

\begin{proof} {\it \,of Proposition \ref{Prop:WH general}}.
The proof goes along the same way as in \cite{GGR(AA)}.
Again we may assume $\phi \in C_c^{\infty}(\Omega)$.
Since $\text{div} \( \frac{x}{|x|^2 \log \frac{aR}{|x|}} \)=\frac{1}{|x|^2 (\log \frac{aR}{|x|})^2}$,
we obtain
\begin{align*}
	\intO \dfrac{|\phi |^2}{|x|^2 (\log \frac{aR}{|x|})^2} d\mu
	&= \intO |\phi |^2 \rho (x) \text{div} \( \frac{x}{|x|^2 (\log \frac{aR}{|x|})} \) dx  \notag \\
	&= - \intO \( 2 \phi \rho \nabla \phi + |\phi|^2 \nabla \rho \) \cdot \frac{x}{|x|^2 \log \frac{aR}{|x|}} dx \notag \\
	&\le 2 \intO \frac{|\phi|}{|x| \log \frac{aR}{|x|}} \left| \nabla \phi + \frac{1}{2} \phi \frac{\nabla \rho}{\rho} \right| d\mu \notag \\
	&\le 2 \( \intO \dfrac{|\phi |^2}{|x|^2 (\log \frac{aR}{|x|})^2} d\mu \)^{\frac{1}{2}} \left\| \nabla \phi + \frac{1}{2} \phi \frac{\nabla \rho}{\rho} \right\|_{L^2(\Omega; d\mu )}.
\end{align*}
Therefore we have
\begin{align}
\label{general rho}
	\frac{1}{4} \intO \dfrac{|\phi |^2}{|x|^2 (\log \frac{aR}{|x|})^2} d\mu \le \left\| \nabla \phi + \frac{1}{2} \phi \frac{\nabla \rho}{\rho} \right\|^2_{L^2(\Omega; d\mu )}.
\end{align}
By applying the same argument as the proof of Proposition 3.1. in \cite{GGR(AA)} with (H2), we obtain
\begin{align}\label{GGR est}
	\left\| \nabla \phi + \frac{1}{2} \phi \frac{\nabla \rho}{\rho} \right\|^2_{L^2(\Omega; d\mu )}
	\le \left[ \frac{4}{1-2\ep}\( \frac{1}{4} +\frac{\eta}{2} \) + 1+\frac{1}{2\eta} \right] \intO |\nabla \phi |^2 d\mu +\frac{C_\ep (1+2\eta )}{2(1-2\ep)} \intO \phi^2 d\mu
\end{align}
for each $\eta >0$, $\ep \in (0, \frac{1}{2})$ and a constant $C_\ep >0$ in (H2).
If we take $\eta=\frac{\sqrt{1-2\ep}}{2}$, then the function $\eta \mapsto \frac{4}{1-2\ep}\( \frac{1}{4} +\frac{\eta}{2} \) + 1+\frac{1}{2\eta}$ 
in (\ref{GGR est}) attains its minimum $\frac{2(1-\ep +\sqrt{1-2\ep})}{1-2\ep}$, which goes to $4$ from above as $\ep \to 0$.
Therefore, from (\ref{general rho}) and (\ref{GGR est}), we get (\ref{WH general}).
\end{proof}
\qed

%
%
\section{Existence and nonexistence of positive solution}

In this section, we consider the following two dimensional Kolmogorov equation perturbed by a singular potential
\begin{align*}
(K_V) \quad 
	\begin{cases}
	\quad \pd_t u(x,t) &= Lu(x,t) + V(x)u(x,t), \quad \,\, t>0, x \in \Omega, \\
	\quad u(x,t) &= 0,  \hspace{9em}  t>0, x \in \pd \Omega, \\
	\quad u(x,0) &= u_0(x), \hspace{7.5em} x \in \Omega
	\end{cases}
\end{align*}
where $\Omega \subseteq \re^2$ is a domain, $0 \in \Omega$, $u_0 \in L^2(\Omega; d\mu)$,
$d\mu = \rho(x) dx$ is a probability Borel measure,
$V \in L_{loc}^1(\Omega)$, $V \ge 0$, and $L$ is the Kolmogorov operator given by
\begin{align*}
	Lu=\lap u +\frac{\nabla \rho}{\rho} \cdot \nabla u.
\end{align*}
Of course if $\Omega = \re^2$, we do not impose the Dirichlet boundary conditions.
Especially, if 
$\rho(x)=\rho_A(x) = c \exp \( -\frac{1}{2} (x^t A x) \)$ and $A$ is a positive definite real $2 \times 2$-symmetric matrix, 
then $L = L_A$ is the symmetric Ornstein-Uhlenbeck operator $L_A u=\lap u -Ax \cdot \nabla u$. 
We define the bottom of the spectrum of $-(L+V)$ to be
\begin{align*}
	\la_1 (L+V):= \inf_{0 \neq \phi \in H_0^1(\Omega; d\mu )} \frac{\intO |\nabla \phi |^2 d\mu - \intO V \phi^2 d\mu}{\intO \phi^2 d\mu}.
\end{align*}

We put the following definition.
\begin{definition}
\label{def weak}
We say that $u$ is a weak solution to $(K_V)$ if for each $T>0$ and any compact subset $K \subset \Omega$, 
we have $u \in C([0,T];L^2(\Omega; d\mu )), Vu \in L^1(K \times (0,T), d\mu dt)$ and 
\begin{align*}
	\int_0^T \intO u (-\pd_t \phi -L\phi ) d\mu dt -\intO u_0 \phi(\cdot, 0) d\mu = \int_0^T \intO V u \phi d\mu dt
\end{align*}
for all $\phi \in W_2^{2,1}(Q_T)$ having compact support with $\phi(\cdot, T)=0$.
Here $Q_T = \Omega \times (0, T)$ and $W_2^{2,1}(Q_T)$ denotes a standard parabolic Sobolev space:
\[
	W_2^{2,1}(Q_T) = \{ u \in L^2(Q_T) \,:\, D_x^{\alpha} u \in L^2(Q_T) \, \text{for} \, |\alpha| \le 2, \, \pd_t u \in L^2(Q_T) \}.
\]
\end{definition}


Let $\Omega$ be a smooth bounded domain in $\re^2$, $0 \in \Omega$, and $R=\sup_{x \in \Omega} |x|$. 
In this case, as in \cite{Cabre-Martel} and \cite{GGR(AA)}, we can obtain the existence and nonexistence result of solutions to $(K_V)$ as follows. 
Since the way of the proof is almost the same as \cite{Cabre-Martel} and \cite{GGR(AA)}, we give here the outline of the proof only.

%
%
\begin{theorem}
\label{Thm:bdd}
Assume that $0<\rho \in C^1(\Omega) \cap C(\overline{\Omega})$ and $0 \le V \in L^1_{\text{loc}}(\Omega)$. 
Then the following assertions hold:

(i) If $\la_1 (L+V) > -\infty$, then for any $u_0 \ge 0$, $u_0 \in L^2(\Omega; d\mu)$,
there exists a positive weak solution $u \in C([0, \infty), L^2(\Omega; d\mu ))$ of $(K_V)$ satisfying 
\begin{align}
\label{decay}
	\| u(t) \|_{L^2(\Omega; d\mu )} \le Me^{\omega t} \| u_0\|_{L^2(\Omega; d\mu)}, \quad t \ge 0
\end{align}
for some constants $M \ge 1$ and $\omega \in \re$.

(ii) If $\la_1 (L+V) = -\infty$, then for any $0 \le u_0 \in L^2(\Omega; d\mu) \setminus \{ 0\}$, there is no positive weak solution of $(K_V)$ satisfying (\ref{decay}).
\end{theorem}

\begin{proof} {\it \,of Theorem \ref{Thm:bdd}}.
\noindent
(i) 
Assume $\la_1 (L+V) > -\infty$ and take $u_0 \ge 0$, $u_0 \not\equiv 0$. 
Set $V_n(x)=\min \{ V(x), n \}$ and $u_{0, n}(x)=\min \{ u_0(x), n \}$. 
Note that since $\rho \in C(\overline{\Omega})$, 
the $L^2$ norm is equivalent to the $L^2_\mu$ norm. 
Consider the following truncated problem $(K_{V_n})$:
\begin{align*}
(K_{V_n}) \quad 
	\begin{cases}
	\quad \pd_t u_n(x,t) &= Lu_n(x,t) + V_n(x)u_n(x,t), \quad \,\, t>0, x \in \Omega, \\
	\quad u_n(x,t) &= 0,  \hspace{9em} \quad  t>0, x \in \pd \Omega, \\
	\quad u_n(x,0) &= u_{0,n}(x), \hspace{7.5em}\quad x \in \Omega.
	\end{cases}
\end{align*}
Since $V_n$ and $u_{0,n}$ are bounded and nonnegative and the drift term $\frac{\nabla \rho}{\rho}$ is also bounded, 
$(K_{V_n})$ admits a unique positive classical solution $u_n$, see e.g. Proposition C.3.2. in \cite{LB book}.
Furthermore $0 < u_n(x,t) \le u_{n+1}(x,t)$ for $n \in \N$ holds on $\Omega \times (0, \infty)$, see e.g. Proposition C.2.3. in \cite{LB book}. 
If we multiply $(K_{V_n})$ by $u_n$ and integrate by parts, we obtain the following in the same way as \cite{GGR(AA)}:
\begin{align*}
	\| u_n(t) \|_{L^2(\Omega; d\mu )} \le e^{-\la_1 (L+V) t} \| u_{0,n} \|_{L^2(\Omega; d\mu)}, \quad t \ge 0,
\end{align*}
which yields that
\begin{align}
\label{proof decay}
	\| u_n(t) \|_{L^2(\Omega; d\mu )} \le e^{-\la_1 (L+V) t} \| u_{0} \|_{L^2(\Omega; d\mu)}, \quad t \ge 0.
\end{align}
By the monotone convergence theorem, we observe that $u_n(t)$ converges to $u(t)$ in $L^2(\Omega; d\mu)$ uniformly for $t \in [0, T]$. 
Since $u_n$ is a weak solution of $(K_{V_n})$, it follows that $u$ is a weak solution of $(K_V)$. The estimate (\ref{decay}) follows from (\ref{proof decay}) and it holds with $M=1$.

\noindent
(ii) 
Assume $\la_1 (L+V) = -\infty$ and assume that there exists a positive solution $u$ of $(K_V)$ with initial data $0 \le u_0 \in L^2(\Omega; d\mu) \setminus \{ 0\}$ satisfying (\ref{decay}). 
We shall derive a contradiction. 
Fix $\phi \in C_c^{\infty}(\Omega )$ with $\int_\Omega \phi^2 dx =1$. 
Let $u_n$ be the unique solution of $(K_{V_n})$ and $v_n$ be the unique solution of 
\begin{align*}
(K_n) \quad 
	\begin{cases}
	\quad \pd_t v_n(x,t) &= Lv_n(x,t), \quad  t>0, x \in \Omega, \\
	\quad v_n(x,t) &= 0,  \qquad  \quad  \, t>0, x \in \pd \Omega, \\
	\quad v_n(x,0) &= u_{0,n}(x), \quad \,\,\, x \in \Omega.
	\end{cases}
\end{align*}
We observe that 
\begin{align}
\label{u_n to v_n}
	u(x,t ) \ge u_n (x, t) \ge v_n(x,t) \ge v_1(x,t), \quad t \ge 0.
\end{align}
It is known that there exists a unique positive function $G_{\Omega} \in C((0,\infty) \times \Omega \times \Omega)$ such that for $u_{0,n} \in C(\ol{\Omega})$,
\begin{align*}
	v_n(t,x)=\intO G_\Omega (t,x,y)u_{0,n}(y) dy, \quad t>0, x \in \Omega,
\end{align*}
see e.g. Proposition C.3.2. in \cite{LB book}. 
Since there exists a ball $B_r$ such that $u_{0,1}(x) > 0$ for $x \in B_r$, we observe that for a.e. $x \in \text{supp} \phi$,
\begin{align*}
	v_1(t,x)&=\intO G_\Omega (t,x,y)u_{0,1}(y) dy \\
	&\ge \( \min_{(x,y) \in \text{supp} \phi \times B_r} G_\Omega (t,x,y) \) \int_{B_r} u_{0,1}(x) dx =: c_r(t; u_{0,1})>0.
\end{align*}
Thus by \eqref{u_n to v_n}, we have $u_n(x,t) \ge c_r(t; u_{0,1})>0$.
If we multiply $(K_{V_n})$ by $\frac{\phi^2}{u_n}$ and integrate by parts, 
then for every $t>1$ we obtain
\begin{align*}
	\int_{\Omega} V_n \phi^2 d\mu \le \pd_t \( \int_{\Omega} (\log u_n(t)) \phi^2d\mu \) + \int_{\Omega} | \nabla \phi |^2 d\mu.
\end{align*}
By integrating from $t=1$ to $t = t$, we have 
\begin{align*}
	(t-1) \int_{\Omega} V_n \phi^2 d\mu \le \int_{\Omega} \( \log \frac{u_n(t)}{u_n(1)} \) \phi^2 d\mu + (t-1) \int_{\Omega} | \nabla \phi |^2 d\mu
\end{align*}
for $t > 1$ and any $n \in \N$, see \cite{GGR(AA)}.
Since there exists a minimal solution $\tilde{u}(t):= \lim_{n \to \infty} u_n(t)$ by (\ref{u_n to v_n}) and the monotone convergence theorem, we obtain 
\begin{align*}
	&\int_{\Omega} V \phi^2 d\mu - \int_{\Omega} | \nabla \phi |^2 d\mu \le \frac{1}{(t-1)} \left[ \int_{\Omega} \( \log \tilde{u}(t) \) \phi^2 d\mu - \int_{\Omega} \( \log \tilde{u}(1) \) \phi^2 d\mu \right] \\
	&\le \frac{1}{(t-1)} \left[  \log (M \| u_0\|_{L^2(\Omega ; d\mu )}) +\omega t +  \log \| \phi \|_\infty - \int_{\Omega} \( \log \tilde{u}(1) \) \phi^2 d\mu \right] \le C <\infty
\end{align*}
in the same way as \cite{GGR(AA)}. 
This contradicts the assumption $\la_1 (L+V) = -\infty$. 
Therefore there is no positive weak solution of $(K_V)$ satisfying (\ref{decay}).
\end{proof}
\qed

As a consequence of Theorem \ref{Thm:bdd} and Remark \ref{rem:WH}, we obtain the main result.

%
%
\begin{theorem}
\label{Thm:heat}
Let $\Omega$ be a bounded domain in $\re^2$, $0 \in \Omega$, $a \ge 1$, and $R=\sup_{x \in \Omega}|x|$.
Assume that $A$ be a positive definite real $2 \times 2$-symmetric matrix and $0 \le V(x) \le \frac{c}{|x|^2(\log \frac{aR}{|x|})^2}$ Then the followings hold:

(i) If \, $0 \le c \le \frac{1}{4}$, then there exists a positive weak solution $u \in C([0, \infty), L^2(\Omega; d\mu_A ))$ of
\begin{align}
\label{drift eq}
\begin{cases}
	\quad \pd_t u(x,t) &= \lap u(x,t) - Ax \cdot \nabla u(x,t) + V(x)u(x,t), \quad t>0, x \in \Omega, \\
	\quad u(x,t) &= 0,  \hspace{15em}  t>0, x \in \pd \Omega, \\
	\quad u(x,0) &= u_0(x), \hspace{13.5em} x \in \Omega,
\end{cases}
\end{align}
satisfying 
\begin{align}\label{decay A}
\| u(t) \|_{L^2(\Omega; d\mu_A )} \le Me^{\omega t} \| u_0\|_{L^2(\Omega; d\mu_A)}, \quad t \ge 0
\end{align}
for some constants $M \ge 1, \omega \in \re$, and any $0 \le u_0 \in L^2(\Omega; d\mu_A)$.

(ii) If $c>\frac{1}{4}$, then for any $0 \le u_0 \in L^2(\Omega; d\mu_A) \setminus \{ 0\}$, there is no positive weak solution of (\ref{drift eq}) with $V(x)=\frac{c}{|x|^2(\log \frac{aR}{|x|})^2}$ satisfying (\ref{decay A}).
\end{theorem}

Furthermore the following result also follows from Proposition \ref{Prop:WH general} and Theorem \ref{Thm:bdd}.

\begin{corollary}\label{general heat}
Let $\Omega$ be a bounded domain in $\re^2$, $0 \in \Omega$, $a \ge 1$, and $R=\sup_{x \in \Omega}|x|$.
Assume that $(H1)-(H2)$ in \S 2 are satisfied and $0 \le V(x) \le \frac{c}{|x|^2(\log \frac{aR}{|x|})^2}$. 
If $c<\frac{1}{16}$, 
then there exists a weak solution $u \in C([0, \infty), L^2(\Omega; d\mu ))$ of $(K_V)$ satisfying 
\begin{align}
\label{decay general}
	\| u(t) \|_{L^2(\Omega; d\mu )} \le Me^{\omega t} \| u_0\|_{L^2(\Omega; d\mu)}, \quad t \ge 0
\end{align}
for some constants $M \ge 1, \omega \in \re$, and any $0 \le u_0 \in L^2(\Omega; d\mu)$.
\end{corollary}

%
%
\section{Weighted Hardy inequality on the half space}

In this section, 
we obtain a weighted $L^p$-Hardy inequality on the half space $\re^N_{+} = \{ x = (x^{\prime}, x_N) \in \re^{N-1} \times \re \, | \, x_N > 0 \}$.
Though the obtained inequality does not have any concrete application in this paper,
we hope it may be also useful to the study of the corresponding parabolic problems.

\begin{theorem}
\label{Thm:WH half}
Let $1<p<\infty$, 
let $A$ be a real $N \times N$-symmetric positive semi-definite matrix, 
and let 
\begin{align*}
	d \mu_A = \rho_A (x) dx \quad \text{with} \quad \rho_A (x)=c \cdot \text{exp} \( -\frac{1}{p} (x^t Ax)^{\frac{p}{2}}\), \,\, x \in \re^N_{+},
\end{align*}
where $c$ is chosen so that $\int_{\re^N_{+}} \rho_A dx = 1$.
Then the inequality
\begin{align}
\label{WH half}
	\( \frac{p-1}{p} \)^p \int_{\re^N_{+}} \dfrac{|u|^p}{x_N^p} d\mu_A \le \int_{\re^N_{+}} \left| \frac{\pd u}{\pd x_N} \right|^p d\mu_A 
	- \( \frac{p-1}{p} \)^{p-1} \int_{\re^N_{+}} \dfrac{|u|^p (x^t Ax)^{\frac{p-2}{2}} (Ax)_N}{x_N^{p-1}} d\mu_A
\end{align}
holds for all $u \in W_{\mu_A, 0}^{1,p}(\re^N_+)$. 
Moreover if $A$ is positive definite, then the constant $(\frac{p-1}{p})^p$ in the left-hand side of \eqref{WH half} is optimal.
\end{theorem}

\begin{proof} {\it \,of Theorem \ref{Thm:WH half}}.
%
It is enough to show that the inequality (\ref{WH half}) holds for all $u \in C_c^1 (\re^N_+)$.
We fix $\la \ge 0$ which will be chosen later. 
Set
\begin{align*}
	\F (x) = \( 0, \cdots, 0, \la \rho_A (x) x_N^{1-p} \)  \quad \text{for} \,\, x \in \re^N_+.
\end{align*}
Then we compute that 
\begin{align*}
	{\rm div} \F (x) = \frac{\pd}{\pd x_N} \( \la \rho_A (x) x_N^{1-p} \) =-\la \rho_A (x) \left[ \frac{p-1}{x_N^p} + \frac{(x^t Ax)^{\frac{p-2}{2}} (Ax)_N}{x_N^{p-1}}  \right].
\end{align*}
By applying integration by parts and Young's inequality, we have
\begin{align*}
	&\int_{\re^N_{+}} |u|^p \la \left[ \frac{p-1}{x_N^p} + \frac{(x^t Ax)^{\frac{p-2}{2}} (Ax)_N}{x_N^{p-1}} \right] \rho_A (x) dx \\
	&=-p \int_{\re^N_+} |u|^{p-2}u \, ( \nabla u \cdot \F ) \, dx \\
	&=-p \la \int_{\re^N_+} \frac{|u|^{p-2}u}{x_N^{p-1}} \( \frac{\pd u}{\pd x_N} \) d\mu_A \\
	&\le \int_{\re^N_+} \left| \frac{\pd u}{\pd x_N} \right|^p d\mu_A + (p-1) \la^{\frac{p}{p-1}} \int_{\re^N_+} \dfrac{|u|^p}{x_N^p} d\mu_A,
\end{align*}
which yields that
\begin{align*}
	(p-1) (\la  - \la^{\frac{p}{p-1}} ) \int_{\re^N_+} \dfrac{|u|^p}{x_N^p} d\mu_A 
	\le \int_{\re^N_+} \left| \frac{\pd u}{\pd x_N} \right|^p d\mu_A - \la \int_{\re^N_+} \dfrac{|u|^p (x^t Ax)^{\frac{p-2}{2}} (Ax)_N }{x_N^{p-1}} d\mu_A.
\end{align*}
If we choose $\la =\( \frac{p-1}{p} \)^{p-1}$ which attains the maximum of the function $\la \mapsto (\la  - \la^{\frac{p}{p-1}} )$ on the half line $[0, \infty )$, 
then we obtain the inequality (\ref{WH half}) for all $u \in C_c^1 (\re^N_+)$. 
Therefore the inequality (\ref{WH half}) also holds for all $u \in W_{\mu_A, 0}^{1,N}(\re^N_+)$ by density.

%
%
Next we show the optimality of the constant $(\frac{p-1}{p})^p$ in (\ref{WH half}) if $A$ is positive definite. 
To do so, we fix $\la> (\frac{p-1}{p})^p$ and take any $\tilde{\la} \in \re$. 
Set 
\begin{align*}
	&E(u)=\int_{\re^N_{+}} \left| \frac{\pd u}{\pd x_N} \right|^p d\mu_A - \tilde{\la} \int_{\re^N_{+}} \dfrac{|u|^p (x^t Ax)^{\frac{p-2}{2}} (Ax)_N}{x_N^{p-1}} d\mu_A, \\
	&F_\la (u)=\la \int_{\re^N_{+}} \dfrac{|u|^p}{x_N^p} d\mu_A
\end{align*}
for $u \in W_{\mu_A, 0}^{1,N}(\re^N_+ ) \setminus \{ 0\}$.
Now we consider a test function $\varphi_{\gamma, \ep} \in W_{\mu_A, 0}^{1,p}(\re^N_+)$ given by
\begin{align*}
	\varphi_{\gamma, \ep} (x) = x_N^\gamma \, \xi_\ep (x_N) \, \xi_\ep (|x'|),
\end{align*}
where $x=(x', x_N) \in \re^{N-1} \times \re_+$, 
$\gamma > \frac{p-1}{p}$, 
$\ep>0$ will be chosen later independent of $\gamma$, 
and $\xi_\ep$ is a cut-off function defined by
\begin{align*}
	\xi_\ep (t)=
	\begin{cases}
		1, \,\,\,&\text{if} \,\,\,0\le t \le \frac{\ep}{2},  \\
		\frac{2}{3\ep} (2\ep -t),  &\text{if} \,\,\, \frac{\ep}{2} < t < 2\ep, \\
		0, &\text{if} \,\,\, t \ge 2 \ep.
	\end{cases}
\end{align*}
Note that there exist $\alpha_1, \alpha_2 >0$ such that $\alpha_1 |x|^2 \le x^t Ax \le \alpha_2 |x|^2$ for all $x \in \re^N_+$ because $A$ is positive definite.
Then we have
\begin{align*}
	F_\la (\varphi_{\gamma, \ep}) 
	&\ge \la c \int_{x_N=0}^{\frac{\ep}{2}} \int_{|x'|\le \frac{\ep}{2}} x_N^{\gamma p-p} \exp \( -\frac{\alpha_2^{\frac{p}{2}}}{p} |x|^p \) dx' dx_N \\
	&+ \la c \int_{x_N=0}^{\frac{\ep}{2}} \int_{\frac{\ep}{2} \le |x'| \le 2\ep} x_N^{\gamma p-p} \xi_\ep (|x'|)^p \exp \( -\frac{\alpha_2^{\frac{p}{2}}}{p} |x|^p \) dx' dx_N,
\end{align*}
which yields that
\begin{align}
\label{deno}
	F_\la (\varphi_{\gamma, \ep}) 
	&\ge \la \frac{c \,|\, B^{N-1}(\frac{\ep}{2}) \,|}{\gamma p - p +1} \( \frac{\ep}{2} \)^{\gamma p -p +1} \exp \( -\frac{\alpha_2^{\frac{p}{2}}}{p} \( \frac{\ep}{\sqrt{2}}\)^p \) \notag \\
	&+ \la \frac{c \,C(\ep)}{\gamma p -p +1} \( \frac{\ep}{2} \)^{\gamma p - p +1} \exp \( -\frac{\alpha_2^{\frac{p}{2}}}{p} \( \frac{3}{2} \ep \)^p \).
\end{align}
Here $|B^{N-1}(r)| = \int_{|x'| \le r} dx'$ denotes the volume of the $(N-1)$-dimensional ball with radius $r$, and
\begin{align*}
	C(\ep) = \int_{\frac{\ep}{2} \le |x'| \le 2\ep} \xi_\ep (|x'|)^p dx'.
\end{align*}
On the other hand, we obtain
\begin{align*}
	\int_{\re^N_+} \left| \frac{\pd \varphi_{\gamma, \ep}}{\pd x_N} \right|^p d\mu_A 
	&\le \gamma^p c \int_{x_N=0}^{\frac{\ep}{2}} \int_{|x'|\le \frac{\ep}{2}} x_N^{\gamma p-p} \exp \( -\frac{\alpha_1^{\frac{p}{2}}}{p} |x|^p \) dx' dx_N \\
	&+ \gamma^p c \int_{x_N=0}^{\frac{\ep}{2}} \int_{\frac{\ep}{2} \le |x'| \le 2\ep} x_N^{\gamma p-p} \xi_\ep (|x'|)^p \exp \( -\frac{\alpha_1^{\frac{p}{2}}}{p} |x|^p \) dx' dx_N \\
	&+ c \int_{x_N=\frac{\ep}{2}}^{2\ep} \int_{|x'| \le 2\ep} \left| \frac{\pd}{\pd x_N} \( x_N^{\gamma} \xi_\ep (x_N) \) \right|^p \xi_\ep (|x'|)^p \exp \( -\frac{\alpha_1^{\frac{p}{2}}}{p} |x|^p \) dx' dx_N, 
\end{align*}
which yields that
\begin{align}
\label{derivative2}
	\int_{\re^N_+} \left| \frac{\pd \varphi_{\gamma, \ep}}{\pd x_N} \right|^p d\mu_A 
	&\le \gamma^p \frac{c \,|\, B^{N-1}(\frac{\ep}{2}) \,|}{\gamma p -p +1} \( \frac{\ep}{2} \)^{\gamma p -p +1}  \notag \\
	&+ \gamma^p \frac{c \,C(\ep)}{\gamma p -p +1} \( \frac{\ep}{2} \)^{\gamma p - p +1} \exp \( -\frac{\alpha_1^{\frac{p}{2}}}{p} \( \frac{\ep}{2}\)^p \) \notag \\
	&+ c 2^{p-1} \( \gamma^p \,D(\ep) + E(\ep) \) \exp \( -\frac{\alpha_1^{\frac{p}{2}}}{p} \( \frac{\ep}{2} \)^p \)
\end{align}
where
\begin{align*}
	&D(\ep) = \int_{x_N=\frac{\ep}{2}}^{2\ep} \int_{|x'| \le 2\ep} x_N^{\gamma p-p} \xi_\ep (x_N)^p \xi_\ep (|x'|)^p dx_N dx', \\ 
	&E(\ep) = \int_{x_N=\frac{\ep}{2}}^{2\ep} \int_{|x'| \le 2\ep} x_N^{\gamma p} \( \frac{2}{3\ep} \)^p \xi_\ep (|x'|)^p dx_N dx'. 
\end{align*}
Note that 
\begin{align}\label{order 1}
	D(\ep) &\le |B^{N-1}(2\ep)| \int_{x_N=\frac{\ep}{2}}^{2\ep} x_N^{\gamma p-p} \xi_\ep (x_N)^p dx_N \notag \\
	&\le C \ep^{N-1} \int_{t=\frac{\ep}{2}}^{2\ep} t^{\gamma p-p} dt \notag \\
	&= C \ep^{\gamma p -p +N}
\end{align}
and 
\begin{align}\label{order 2}
	E(\ep) \le |B^{N-1}(2\ep)| \int_{x_N=\frac{\ep}{2}}^{2\ep} x_N^{\gamma p} \( \frac{2}{3\ep} \)^p dx_N = C \ep^{\gamma p -p +N}
\end{align}
for some absolute value $C > 0$.
In the same way as above, we also obtain
\begin{align}
\label{A half}
	\left| \int_{\re^N_{+}} \dfrac{|\varphi_{\gamma, \ep}|^p (x^t Ax)^{\frac{p-2}{2}} (Ax)_N}{x_N^{p-1}} d\mu_A \right| 
	&\le \alpha_2^{\frac{p-2}{2}} \int_{x_N= 0}^{2\ep} \int_{|x'| \le 2\ep} x_N^{\gamma p-p +1} | A | |x|^{p-1} \rho_A(x) dx \notag \\
	&\le C \ep^{\gamma p + N}
\end{align}
for $\gamma > \frac{p-1}{p}$ sufficiently close to $\frac{p-1}{p}$.
From (\ref{derivative2}), (\ref{order 1}), (\ref{order 2}), and \eqref{A half}, we have
\begin{align}\label{nume}
	E(\varphi_{\gamma, \ep}) 
	&\le \gamma^p \frac{c \,|\, B^{N-1}(\frac{\ep}{2}) \,|}{\gamma p - p +1} \( \frac{\ep}{2} \)^{\gamma p -p +1} \notag \\
	&+ \gamma^p \frac{c \,C(\ep)}{\gamma p - p +1} \( \frac{\ep}{2} \)^{\gamma p -p +1} \exp \( -\frac{\alpha_1^{\frac{p}{2}}}{p} \( \frac{\ep}{2}\)^p \) + C
\end{align}
for an absolute value $C >0$ when $\gamma > \frac{p-1}{p}$ and $\ep \in (0,1)$.

We have chosen a small $\ep >0$ in advance such that
\begin{align*}
	\la \cdot \min \left\{  \exp \( -\frac{\alpha_2^{\frac{p}{2}}}{p} \( \frac{\ep}{\sqrt{2}} \)^p \), 
	\exp \( -\frac{\alpha_2^{\frac{p}{2}}}{p} \( \frac{3}{2} \ep \)^p + \frac{\alpha_1^{\frac{p}{2}}}{p} \( \frac{\ep}{2}\)^p \) \right\} > \( \frac{p-1}{p} \)^p,
\end{align*}
which is possible since $\la > \( \frac{p-1}{p} \)^p$.
For this choice of $\ep$, 
we may take $\gamma > \frac{p-1}{p}$ sufficiently close to realize that
\begin{equation}
\label{gamma choice}
	\begin{cases}
	&\la \exp \( -\frac{\alpha_2^{\frac{p}{2}}}{p} \( \frac{\ep}{\sqrt{2}} \)^p \) > \gamma^p, \\
	&\la \exp \( -\frac{\alpha_2^{\frac{p}{2}}}{p} \( \frac{3}{2} \ep \)^p \) > \gamma^p \exp \( -\frac{\alpha_1^{\frac{p}{2}}}{p} \( \frac{\ep}{2}\)^p \)
	\end{cases}
\end{equation}
holds true.
From the estimates (\ref{deno}) and (\ref{nume}), we observe that
\begin{align*}
	E(\varphi_{\gamma, \ep}) < F_\la (\varphi_{\gamma, \ep})
\end{align*}
for $\gamma > \frac{p-1}{p}$ sufficiently close to $\frac{p-1}{p}$ satisfying \eqref{gamma choice}.
Therefore the inequality (\ref{WH half}) never holds if the constant on the left-hand side of (\ref{WH half}) is larger than $(\frac{p-1}{p})^p$.
Hence the constant $(\frac{p-1}{p})^p$ in (\ref{WH half}) is optimal.
\end{proof}
\qed

%
%


%
%

\begin{acknowledgements}
Part of this work was supported by 
JSPS Grant-in-Aid for Fellows(DC2), No.16J07472 (M.S), 
and
JSPS Grant-in-Aid for Scientific Research (B), No.15H03631, 
\end{acknowledgements}

\end{document}